\documentclass[12pt]{article}
\usepackage{amsmath, amsfonts,amsthm,amssymb,amsbsy,upref,color,graphicx,amscd,hyperref,makeidx,enumerate}
\usepackage{tikz}
\usepackage[active]{srcltx}
\usepackage[latin1]{inputenc}
\usepackage{enumerate}
\usepackage[english]{babel}
\usepackage{newlfont}
\usepackage{amsbsy}
\usepackage[english]{babel}
\usepackage[center]{caption2}
\usepackage{amssymb}
\usepackage{amsmath}
\usepackage{latexsym}
\usepackage{amsthm}
\usepackage{amsthm}
\theoremstyle{plain}
\newtheorem{thm}{Theorem}

\newtheorem{prop}[thm]{Proposition}
\newtheorem{nota}[thm]{Notation}
\newtheorem{rem}[thm]{Remark}
\newtheorem{defin}[thm]{Definition}

\newcommand{\R}{\mathbb{R}}

\newcommand{\N}{\mathbb{N}}

\usepackage{mathtools}
\def\multiset#1#2{\ensuremath{\left(\kern-.2em\left(\genfrac{}{}{0pt}{}{#1}{#2}\right)\kern-.2em\right)}}
\usepackage[center]{caption2}

\begin{document}

\title{Families of $2$-weights of some particular graphs}
\author{Agnese Baldisserri, Elena Rubei}
\date{}
\maketitle

\def\thefootnote{}
\footnotetext{ \hspace*{-0.36cm}
{\bf 2010 Mathematical Subject Classification: 05C05, 05C12, 05C22} 

{\bf Key words: weighted graphs, $2$-weights} }

\begin{abstract}
Let ${\cal G}=(G,w) $ be a positive-weighted graph, that is a graph $G$ endowed with a function $w$ from the edge set of $G$ to the set of positive real numbers; for any distinct vertices $i,j $, we define $D_{i,j}({\cal G})$ to be  the weight of the  path  in $G$ joining $i$ and $j$ with minimum weight. In this paper  we fix a particular class of graphs and we give a criterion to establish whether, given a family of positive real numbers $\{D_I\}_{I \in { \{1,...., n\}  \choose 2}}$,  there exists a positive-weighted graph ${\cal G} =(G,w) $ in the class we have fixed,  with vertex set equal to $\{1,....,n\}$ and such that $D_I ({\cal G}) =D_I$ for any  $I \in { \{1,...., n\}  \choose 2}$. In particular, the classes of graphs we consider are the following:  snakes,  caterpillars, polygons,  bipartite graphs,  complete graphs,  planar graphs. 
\end{abstract}

\section{Introduction}

For any graph $G$, let $E(G)$, $V(G)$ and $L(G)$ 
 be respectively the set of the edges,   
the set of the vertices and  the set of the leaves of $G$.
A {\bf weighted graph} ${\cal G}=(G,w)$ is a graph $G$ 
endowed with a function $w: E(G) \rightarrow \R$. 
For any edge $e$, the real number $w(e)$ is called the weight of the edge 
and, for any subgraph $H$ of  $G$, we define  $w(H)$  to be the sum of the weights of the edges of $H$. We say that ${\cal G}$ is {\bf positive-weighted}  if all the weights of the edges are positive. 

\begin{defin}
Let ${\cal G}=(G,w) $ be a positive-weighted graph; for any distinct $i,j \in V(G)$ we define 
$$D_{\{i,j\}}({\cal G})= \min \{w(H) \,| \, H \; \mbox{a  connected subgraph of } G \mbox{  such that } V(H) \ni i,j\} .$$
 More simply, we denote 
$D_{\{i,j\}}({\cal G})$ by $D_{i,j}({\cal G})$.
We say that a connected  subgraph $H$ {\bf realizes} $ D_{i,j}({\cal G})$
if   $V(H) \ni i,j$ and $w(H) = D_{i,j}({\cal G})$.
Obviously such a subgraph 
must be a path with endpoints $i$ and $j$. 
We call  the  $ D_{i,j}({\cal G})$ the {\bf $2$-weights} of ${\cal G}$.
\end{defin}

Throughout the paper we will consider only simple finite  connected graphs.
Observe that in the case $G$ is a tree,  $ D_{i,j}({\cal G})$ is  the weight of the unique path joining $i$ and $j$.

 If $S $ is a subset of $V(G)$, the $2$-weights $ D_{i,j}({\cal G})$ with $i,j \in S$ give
a vector in $\mathbb{R}^{  S \choose 2}$. This vector is called 
$2${\bf -dissimilarity vector} of $({\cal G}, S)$.
Equivalently,  we can speak 
of the {\bf family of the $2$-weights}  of $({\cal G}, S)$.

We can wonder when a family of real numbers is the family of the $2$-weights of some positive-weighted  graph and of some subset of the set of its vertices.
If $S$ is a finite set of cardinality greater than $2$,
we say that a family of  real numbers
 $\{D_{I}\}_{I \in {S \choose 2}}$   is  p-graphlike if  there exist
a positive-weighted graph ${\cal G}=(G,w)$ and a subset $S$ 
of the set of its vertices such that $ D_{I}({\cal G}) = D_{I}$  for any 
 $2$-subset $I$ of $ S$.
If the graph is a  positive-weighted  tree
 ${\cal T}=(T,w)$ we say that the family 
 is  p-treelike.

Weighted graphs have applications in several disciplines, such as 
biology, psychology,  archeology, engineering.  
Phylogenetic trees are positive-weighted trees
whose vertices represent  species and the weight of an edge is given by how much the DNA sequences of the species represented by the vertices of the edge differ; 
also archeologists represent evolutions of manuscripts by positive-weighted trees. 
There is a wide literature concerning graphlike dissimilarity families
and treelike dissimilarity families,
in particular concerning methods to reconstruct weighted trees from their dissimilarity families; these methods are used by biologists to reconstruct phylogenetic trees. 
See for example \cite{Dresslibro}, \cite{S-S} for overviews on phylogenetic trees.
Weighted graphs can represent hydraulic webs or railway webs
where the weight of an edge is given by the length or the cost 
  (or the difference between the  earnings and the cost) of the line represented by that edge. 


The first contribution to the 
characterization of  graphlike families of numbers dates back  to 1965 and it is due  to Hakimi and Yau, see  \cite{H-Y}:

\begin{thm} \label{HY} {\bf (Hakimi-Yau)}
A family of positive real numbers
$\{D_{I}\}_{I \in {\{1,...,n\} \choose 2}}$
is p-graphlike if and only if the $D_I$ satisfy the 
triangle inequalities, i.e. if and only if $D_{i,j} \leq D_{i,k} + D_{k,j}$ for any distinct $i,j,k \in [n]$.
\end{thm}

In the same years, also a criterion for a metric on
a finite set to be p-treelike
was established, see \cite{B}, \cite{SimP}, \cite{Za}:

\begin{thm} \label{Bune} {\bf (Buneman-SimoesPereira-Zaretskii)}
Let $\{D_{I}\}_{I \in {\{1,...,n\} \choose 2}}$ be a family  of positive real numbers 
satisfying the triangle inequalities.
It is p-treelike  
 if and only if the $D_I$ satisfy the so-called 
four-point condition, i.e., for all distinct $i,j,k,h  \in \{1,...,n\}$,
the maximum of $$\{D_{i,j} + D_{k,h},D_{i,k} + D_{j,h},D_{i,h} + D_{k,j}
 \}$$ is attained at least twice. 
\end{thm}

Also   the case of not necessarily nonnegative weights has been studied, see \cite{H-P} for graphs  and  \cite{B-S} for trees.

Finally we want to mention that recently 
$k$-weights of weighted graphs for $k \geq 3$ have been introduced and studied; in particular there are some results concerning the characterization of families of $k$-weights,
see for instance  \cite{B-R}, \cite{B-R2},   \cite{H-H-M-S}, \cite{Iri}, 
\cite{L-Y-P}, \cite{P-S}. 
The study of  $k$-weights for $k \geq 3$ 
is motivated by the fact that they are more reliable statistically than $2$-weights
and so the reconstruction of weighted trees from them can be more accurate than the reconstruction from $2$-weights.

In this paper, we fix a particular class of graphs,
we consider a family of positive real numbers
 $\{D_I\}_{I \in { \{1,...., n\}  \choose 2}}$ and  
 we give a criterion to establish whether  there exists a positive-weighted graph 
${\cal G} =(G,w) $ in the class we have fixed,  with $V(G)=\{1,....,n\}$ and such that $D_I ({\cal G}) =D_I$ for any  
$I \in { \{1,...., n\}  \choose 2}$.
In particular the classes we consider are the following:  snakes,  caterpillars, polygons,  bipartite graphs,  complete graphs,  planar graphs.

\section{Some definitions and some remarks}

\begin{nota}
$\bullet $ Let $ \mathbb{R}_{+} =\{x \in \mathbb{R} | \; x >0\}$.

$ \bullet $ Throughout the paper let $n \in \N -\{0,1\}$ and let $[n]= \{1,..., n\}$.

$ \bullet $ If $G$ is a graph, for any $v,v' \in V(G)$, let $e(v,v') $ denote the edge joining $v$ and $v'$.

$ \bullet $ For any  graph $G$, let $V^i(G)$ be the set of the vertices of $G$ of degree $i$ and let  $V^{\geq i} (G) = \cup_{j \geq i } V^j(G)$. 

$\bullet $   For any  graph $G$, we say that an edge is {\bf pendant}
if it is incident to a leaf.

$ \bullet $ If $T$ is a tree, for any $v,v' \in V(T)$, let $p(v,v') $ denote the unique path joining $v$ and $v'$.

$\bullet$ Throughout the paper, the word ``tree'' will denote a finite  tree and the word ``graph'' will denote a finite connected graph.

$\bullet $  Let $n \in \N-\{0,1\}$. For any  family of  real numbers  
$\{D_I\}_{I \in {[n] \choose 2}}$, we denote $D_{\{i,j\}}$ by $D_{i,j}$.
\end{nota}

\begin{defin}
\begin{itemize}
\item A \textbf{snake}  is a tree with only $2$ leaves.
\item We say that a tree $C$ is a \textbf{caterpillar}  if there is a  path $S$ such that $V(S)= V^{\geq 2} (C)$. 
We call $S$ the \textbf{spine} of the caterpillar. 
\item 
A graph $P$ with $n$ vertices  is a \textbf{polygon} if we can rename the vertices by $i_1,....,i_n$ in such way that $E(P)=\{ e(i_1,i_2) ,....., e(i_{n-1}, i_n) , e(i_n ,i_1) \}$.

\item A graph $G$ is \textbf{complete} if $ E(G)$ contains the edge  $e(i,j) $ for any $i,j \in V(G)$. The complete graph with $n$ vertices is usually denoted by $K_n$.
\item a graph $B$ is a \textbf{bipartite graph} on two subsets $X$ and $Y$ of $V(B)$ if: 
\begin{itemize}
\item[-] $X \cap Y = \emptyset$;
\item[-] $X \cup Y = V(B)$;
\item[-] $E(B) \subset \{e(x,y) \, | \, x \in X, \, y \in Y\}$.
\end{itemize}
\item A  \textbf{bipartite graph} $B$ on two subsets $X$ and $Y$ of $V(B)$ is \textbf{complete} if $E(B)$ contains the edge  $e(i,j) $ for any $i \in X, j \in Y$. The complete bipartite graph on two sets one of cardinality $m$ and one of cardinality   $n$ is usually denoted by $K_{m,n}$. 
\end{itemize}
\end{defin}

\begin{rem} Let $B$ be a graph and let 
$X, Y, P, Q \subset V(B)$ such that $B$ is a bipartite graph on $X$ and $Y$ and is a bipartite graph on $P$ and $Q$; then we can 
 easily   show that
  $X=P$ and $Y=Q$ or vice versa.
\end{rem}

\begin{defin}
Let ${\cal G}=(G,w)$ a positive-weighted graph, we say that an edge $e$ of $G$ is \textbf{useful} if there exist  $i,j \in V(G)$ such that all the paths realizing the $2$-weight $D_{i,j}({\cal G})$ contain the edge $e$. We say that an edge $e$ is \textbf{useless} if it is not useful, that is, if all the $2$-weights of the graph are realized  by at least a path which do not contain $e$. Finally, we say that a graph ${\cal G}$ is \textbf{pruned} if all its edges are useful.  
\end{defin}

\begin{defin}
Let $\{D_I\}_{I \in {[n] \choose 2}}$ be a family  in $\R_{+}$.
 We say that the family is
\begin{itemize}
\item \textbf{snakelike} if there exists a positive-weighted snake ${\cal S}=(S,w)$ with $V(S)=[n]$ such that $D_{i,j}({\cal S})=D_{i,j}$ for any $i,j \in [n]$,
\item \textbf{caterpillarlike} if there exists a positive-weighted caterpillar ${\cal C}=(C,w)$ with $V(C)=[n]$ such that $D_{i,j}({\cal C})=D_{i,j}$ for any $i,j \in [n]$,
\item \textbf{polygonlike} if there exists a positive-weighted polygon ${\cal P}=(P,w)$ with $V(P)=[n]$ such that $D_{i,j}({\cal P})=D_{i,j}$ for any $i,j \in [n]$,
\item \textbf{co-graphlike} if there exists a positive-weighted complete graph ${\cal G}=(G,w)$ with $V(G)=[n]$ such that $D_{i,j}({\cal G})=D_{i,j}$ for any $i,j \in [n]$ and $G$ is pruned,
\item \textbf{bigraphlike} on two subsets $X$ and $Y$ of $[n]$ if there exists a positive-weighted bipartite graph ${\cal B}=(B,w)$ with $V(B)=[n]$ on $X,Y $ such that $D_{i,j}({\cal B})=D_{i,j}$ for any $i,j \in [n]$,
\item \textbf{co-bigraphlike} if there exists a pruned positive-weighted complete bipartite graph ${\cal B}=(B,w)$ with $V(B)=[n]$ such that $D_{i,j}({\cal B})=D_{i,j}$ for any $i,j \in [n]$,
\item \textbf{planargraphlike} if there exists a positive-weighted planar graph ${\cal G}=(G,w)$ with $V(G)=[n]$ such that $D_{i,j}({\cal G})=D_{i,j}$ for any $i,j \in [n]$.
\end{itemize} 
To be precise we should say ``p-snakelike, p-caterpillarlike....''
to point out that we are considering positive-weighted graphs, but, since we will consider only positive-weighted graphs and so no confusion can arise, for semplicity we will omit the letter ``p''.
\end{defin}

\begin{rem}\label{rem2}
Let ${\cal G}=(G,w)$ be a positive-weighted graph with $V(G)=[n]$ and let $i,j \in [n]$; if $D_{i,j}({\cal G})$ is realized by a path $H$ in $G$, then, for any $k,t \in V(H)$, the $2$-weight $D_{k,t}({\cal G})$ is realized by the path   in $H$ with endpoints $k$ and $t$.
\end{rem}

\begin{proof}
Suppose, contrary to our claim, that there exist $k,t \in V(H)$ such that any path realizing $D_{k,t}({\cal G})$ is not contained in $H$. Call $J$ one of the paths realizing $D_{k,t}({\cal G})$  and call $H'$ a path joining $k$ with $t$ contained in $H$. Then we would have:
\begin{equation} \label{sss}
w(H') > D_{k,t}({\cal G}) =w(J);
\end{equation}
moreover,
$$w(H) = D_{i,j}({\cal G}) \leq w((H \smallsetminus H') \cup J) \leq w(H \smallsetminus H') + w(J);$$
thus 
$$w(H') \leq w(J),$$
which is absurd because it contradicts (\ref{sss}).
\end{proof}

\begin{defin} Let $\{D_I\}_{I \in {[n] \choose 2}}$ be a family  in $\R_{+}$ satisfying the triangle inequalities; we say that an element of the family $D_{i,j}$ is \textbf{indecomposable} if $D_{i,j}<D_{i,z}+D_{z,j}$ for any $z \in [n]\smallsetminus \{i,j\}$.
\end{defin}

\begin{rem}\label{latounico}
Let ${\cal G}=(G,w)$ be  a positive-weighted graph 
 such that $V(G)=[n]$. For any  $i,j \in [n]$, the $2$-weight $D_{i,j}(\cal{G})$ is indecomposable if and only if $E(G)$ contains the edge $e(i,j)$ and $e(i,j) $ is useful. In this case we have that the $2$-weight $D_{i,j}(\cal{G})$ is realized only by the edge $e(i,j)$ and, in particular, $D_{i,j}(\cal{G})=w(e(i,j))$.
\end{rem}

\begin{proof}
Suppose that $D_{i,j}(\cal{G})$ is indecomposable; if it were realized by a path joining $i$ with $j$ different from $e(i,j)$, it would contain another vertex $z \in [n] \smallsetminus \{i,j\}$, then, by Remark \ref{rem2}, we would have that $D_{i,j}(\cal{G})=D_{i,z}(\cal{G})+D_{z,j}(\cal{G})$, which is absurd. So $D_{i,j}(\cal{G})$ can be realized only by $e(i,j)$, thus  $e(i,j) \in E(G)$ and $e(i,j)$ useful. Conversely, suppose to have a useful edge $e(i,j) \in E(G)$: by definition, there exist two vertices $a,b \in V(G)$ such that all the paths realizing $D_{a,b}({\cal G})$ contain $e(i,j)$. Then the $2$-weight $D_{i,j}(\cal{G})$ is realized by the edge $e(i,j)$ (by Remark \ref{rem2}) and it can be realized only 
by the edge $e(i,j)$, so it is indecomposable. 
\end{proof}


\section{Snakes and caterpillars}

In this section we give a characterization of snakelike families  in $ \mathbb{R}_{+}$ and  a characterization of caterpillarlike ones. 

\begin{thm}\label{thm:snakelike}
Let $\{D_I\}_{I \in {[n] \choose 2}}$ be a family  in $\R_{+}$
 and let $x,y \in [n]$ such that $D_{x,y}=max_{i,j \in [n]}\{D_{i,j}\}$; the family is snakelike if and only if $D_{i,j}=|D_{i,x}-D_{j,x}|$ for any distinct $i, j \in [n] \smallsetminus \{x\}$.
\end{thm}  

\begin{proof} $\Longrightarrow $ Very easy to prove.

$\Longleftarrow $ 
First note  that  $D_{a,x} \neq D_{b,x}$ for any distinct $a,b \in [n] \smallsetminus \{x\}$; otherwise we would have: 
$$D_{a,b}=|D_{a,x}-D_{b,x}|=0,$$ 
which is absurd because, by assumption, the elements of the family are positive. 
Let us denote the elements of $[n] \smallsetminus \{x,y\}$ by 
 $i_1,i_2,...,i_{n-2}$ in  such a way that 
 $$D_{i_j,x}<D_{i_{j+1},x}$$ for any $j =1,....,n-3$
and let  ${\cal S}= (S,w) $ be the positive-weighted snake  defined
as follows (see  Figure \ref{fig:snake}):
let $S$ be the  snake with $V(S)=[n]$ and $E(S)=
 \{ e(x,i_1), e(i_1,i_2),......, e(i_{n-2},y) \}$ and 
define the  weights of $ e(x,i_1), e(i_1,i_2),......, e(i_{n-2},y)$ to be, respectively, 
$D_{i_1,x}$, $D_{i_2,x}-D_{i_1,x}$, $D_{i_3,x}-D_{i_2,x}$,......,$D_{y,x}-D_{i_{n-2},x}$.

\begin{figure}[h!]
\begin{center}
\begin{tikzpicture}
\draw [thick] (0,0) --(12,0);

\node [below] at (0,0) {$\mathbf{x}$};
\node [below] at (2,0) {$\mathbf{i_1}$};
\node [below] at (4,0) {$\mathbf{i_2}$};
\node [below] at (10,0) {$\mathbf{i_{n-2}}$};
\node [below] at (12,0) {$\mathbf{y}$};

\node [above] at (1,0) {\footnotesize $D_{i_1,x}$};
\node [above] at (3,0) {\footnotesize $D_{i_2,x}-D_{i_1,x}$};
\node [above] at (11,0) {\footnotesize $D_{y,x}-D_{i_{n-2},x}$};

\draw [fill] (0,0) circle [radius=0.06];
\draw [fill] (2,0) circle [radius=0.06];
\draw [fill] (4,0) circle [radius=0.06];
\draw [fill] (10,0) circle [radius=0.06];
\draw [fill] (12,0) circle [radius=0.06];
\draw [fill] (6,0) circle [radius=0.06];
\end{tikzpicture}
\caption{a positive-weighted snake realizing the family $\{D_{i,j}\}${\label{fig:snake}}}
\end{center}
\end{figure}
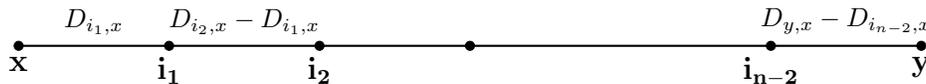

We have to check that $D_{i,j}({\cal S})=D_{i,j}$ for any $i,j \in [n]:$

\begin{itemize}
\item[-] $D_{i,x}({\cal S})=D_{i,x}$ for any $i \in [n] \smallsetminus \{x\}$ by construction;
\item[-] $D_{i,j}({\cal S})=|D_{i,x}({\cal S})-D_{j,x}({\cal S})|= |D_{i,x}-D_{j,x}|= D_{i,j}$ for any $i,j \in [n] \smallsetminus \{x\}$, where the last equality holds by assumption.
\end{itemize}
\end{proof}

Before 
studying caterpillarlike families, we introduce a definition and we state a theorem that will be useful later:

\begin{defin}\label{medianfamily}
Let  $\{D_I\}_{I \in {[n] \choose 2}}$ be a family in $\R_{+}$. We say that the family $\{D_I\}$ is a {\bf median family} if, for any $a,b,c \in [n]$, there exists a unique element $m \in [n]$ such that 
$$D_{i,j}=D_{i,m}+D_{j,m}$$
for any distinct  $i,j \in \{a,b,c\}.$
\end{defin}

Observe that a median family satisfies the triangle inequalities.
The theorem below, probably well-known to experts,  was suggested to us by an anonymous referee in October 2014; later we have found it also
in \cite{H-F};
we defer to \cite{B-R5} for a shorter proof.

\begin{thm}\label{thm:median}
Let $\{D_I\}_{I \in {[n] \choose 2}}$ be a family  in $\R_{+}$. There exists a positive-weighted tree ${\cal T}=(T,w)$, with $V(T)=[n]$, such that $D_I(\cal{T})=D_ I$ for all $I \in {[n] \choose 2}$ if and only if the four-point condition holds and  the family $\{D_I\}_I$ is median.
\end{thm}

\vspace{0.5cm}

Now, consider a positive-weighted caterpillar ${\cal C}=(C,w)$ with $V(C)=[n]$.
 
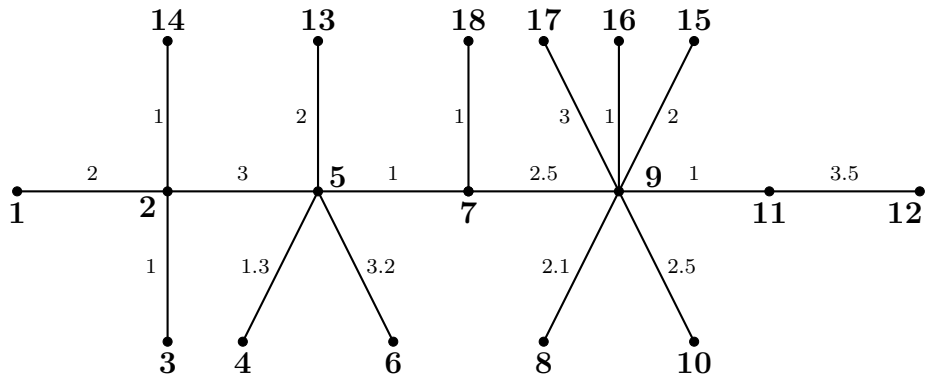
\begin{figure}[h!]
\begin{center}
\begin{tikzpicture}
\draw [thick] (0,0) --(12,0);
\draw [thick] (2,0) --(2,2);
\draw [thick] (2,0) --(2,-2);
\draw [thick] (4,0) --(4,2);
\draw [thick] (4,0) --(5,-2);
\draw [thick] (4,0) --(3,-2);
\draw [thick] (6,0) --(6,2);
\draw [thick] (8,0) --(8,2);
\draw [thick] (8,0) --(7,2);
\draw [thick] (8,0) --(9,2);
\draw [thick] (8,0) --(7,-2);
\draw [thick] (8,0) --(9,-2);

\node [below] at (0,0) {$\mathbf{1}$};
\node [left] at (2,-0.2) {$\mathbf{2}$};
\node [below] at (2,-2) {$\mathbf{3}$};
\node [below] at (3,-2) {$\mathbf{4}$};
\node [right] at (4,0.2) {$\mathbf{5}$};
\node [below] at (5,-2) {$\mathbf{6}$};
\node [below] at (6,0) {$\mathbf{7}$};
\node [below] at (7,-2) {$\mathbf{8}$};
\node [right] at (8.2,0.2) {$\mathbf{9}$};
\node [below] at (9,-2) {$\mathbf{10}$};
\node [below] at (10,0) {$\mathbf{11}$};
\node [below] at (11.8,0) {$\mathbf{12}$};
\node [above] at (9,2) {$\mathbf{15}$};
\node [above] at (8,2) {$\mathbf{16}$};
\node [above] at (7,2) {$\mathbf{17}$};
\node [above] at (6,2) {$\mathbf{18}$};
\node [above] at (4,2) {$\mathbf{13}$};
\node [above] at (2,2) {$\mathbf{14}$};

\node [above] at (1,0) {\scriptsize{$2$}};
\node [above] at (3,0) {\scriptsize{$3$}};
\node [above] at (5,0) {\scriptsize{$1$}};
\node [above] at (7,0) {\scriptsize{$2.5$}};
\node [above] at (9,0) {\scriptsize{$1$}};
\node [above] at (11,0) {\scriptsize{$3.5$}};

\node [left] at (2,-1) {\scriptsize{$1$}};
\node [left] at (3.5,-1) {\scriptsize{$1.3$}};
\node [right] at (4.5,-1) {\scriptsize{$3.2$}};
\node [left] at (7.5,-1) {\scriptsize{$2.1$}};
\node [right] at (8.5,-1) {\scriptsize{$2.5$}};

\node [left] at (2.1,1) {\scriptsize{$1$}};
\node [left] at (4,1) {\scriptsize{$2$}};
\node [left] at (6.1,1) {\scriptsize{$1$}};
\node [left] at (7.5,1) {\scriptsize{$3$}};
\node [left] at (8.1,1) {\scriptsize{$1$}};
\node [right] at (8.5,1) {\scriptsize{$2$}};

\draw [fill] (0,0) circle [radius=0.06];
\draw [fill] (2,0) circle [radius=0.06];
\draw [fill] (4,0) circle [radius=0.06];
\draw [fill] (6,0) circle [radius=0.06];
\draw [fill] (8,0) circle [radius=0.06];
\draw [fill] (10,0) circle [radius=0.06];
\draw [fill] (12,0) circle [radius=0.06];
\draw [fill] (6,0) circle [radius=0.06];
\draw [fill] (2,-2) circle [radius=0.06];
\draw [fill] (3,-2) circle [radius=0.06];
\draw [fill] (5,-2) circle [radius=0.06];
\draw [fill] (7,-2) circle [radius=0.06];
\draw [fill] (9,-2) circle [radius=0.06];
\draw [fill] (2,2) circle [radius=0.06];
\draw [fill] (4,2) circle [radius=0.06];
\draw [fill] (6,2) circle [radius=0.06];
\draw [fill] (7,2) circle [radius=0.06];
\draw [fill] (8,2) circle [radius=0.06];
\draw [fill] (9,2) circle [radius=0.06];

\end{tikzpicture}
\caption{a positive-weighted caterpillar  ${\cal C}=(C,w)$ with $V(C)=[18]$}{\label{fig:caterpillar}}
\end{center}
\end{figure}

Given a vertex $x \in V(C)$, we can define 
$$ t_x = \frac{1}{2} min_{z,y \in V(C)\smallsetminus \{x\}}\{D_{x,y}({\cal C})+D_{x,z}({\cal C})-D_{y,z}({\cal C})\};$$
it is easy to show that, if $x \in L(C)$, then $t_x$ is the weight of the pendant edge associated to $x$ and that $t_x=0$ if and only if $x \notin L(C)$, that is, $x$ belongs to the spine of $C$. 

\begin{rem}\label{codecaterpillar}
Let ${\cal C}=(C,w)$ be a positive-weighted caterpillar with $V(C)=[n]$
and let  $x_1, x_2 \in V(C)$  be  such that $p(x_1,x_2)$ is 
 the spine of $C$. Let $X_1$ (respectively $X_2$) be the set of the leaves of $C$ adjacent to $x_1$ (respectively $x_2$)
(for example, in  Figure \ref{fig:caterpillar}, we have that  $\{x_1, x_2\} =\{2,11\}$ and, 
if we take for instance $ x_1=2$ and
$x_2 =11 $, we have that $X_1=\{1,3,14\}$ and $X_2 =\{12\} $).
 If we consider two vertices $a,b \in V(C)$ such that 
$$D_{a,b}({\cal C}) -t_a -t_b = max_{i,j \in [n]} \{D_{i,j}({\cal C}) -t_i -t_j \},$$
we have that $a \in X_1 \cup \{x_1\} $ and $b \in X_2 \cup\{x_2\} $ or vice versa.
\end{rem}
 
\begin{proof}
It is sufficient to note that, for any $i,j \in V(C)$, $$D_{i,j}({\cal C}) -t_i -t_j =w(p( \overline{i}, \overline{j}) ),$$
 where $\overline{i}$ is defined as follows:
it is  equal to $i$ if $i \notin L(C)$, while it is  the vertex adjacent to $i$ if $i \in L(C)$;
 analogously $\overline{j}$.
\end{proof}

\begin{rem}\label{condizionecaterpillar}
Let ${\cal C}=(C,w)$ a positive-weighted tree with $V(C)=[n]$; call $a,b$ two vertices of $C$ such that  
$$D_{a,b}({\cal C}) -t_a -t_b = max_{i,j \in [n]} \{D_{i,j}({\cal C}) -t_i -t_j \}.$$
The tree $C$ is a caterpillar  if and only if for any $i,j \in [n]\smallsetminus \{a,b\}$ we have that 
\begin{equation}\label{eq:caterpillar}
D_{a,b}({\cal C})+D_{i,j}({\cal C})\geq max \{ D_{a,i}({\cal C})+D_{b,j}({\cal C}), D_{a,j}({\cal C})+D_{b,i}({\cal C})\}.
\end{equation}
\end{rem}

\begin{proof}
If $C$ is a caterpillar, then, using Remark \ref{codecaterpillar} it is easy to check that  (\ref{eq:caterpillar}) holds for any $i,j \in [n]\smallsetminus \{a,b\}$. Now, suppose that $C$ is not a caterpillar, then there must be a vertex $c$ with degree grater than $1$ which is not in  $p(a,b)$. 

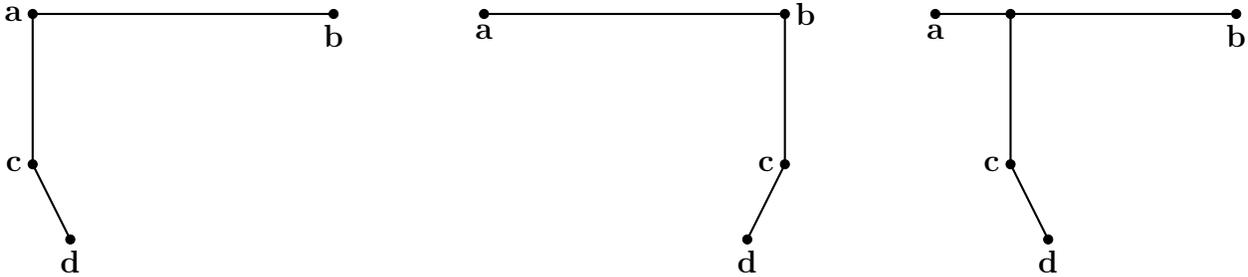
\begin{figure}[h!]
\begin{center}
\begin{tikzpicture}
\draw [thick] (-3,0) --(1,0);
\draw [thick] (-3,0) --(-3,-2);
\draw [thick] (-3,-2) --(-2.5,-3);
\node [left] at (-3,0) {$\mathbf{a}$};
\node [below] at (1,0) {$\mathbf{b}$};
\node [left] at (-3,-2) {$\mathbf{c}$};
\node [below] at (-2.5,-3) {$\mathbf{d}$};
\draw [fill] (-3,0) circle [radius=0.06];
\draw [fill] (1,0) circle [radius=0.06];
\draw [fill] (-3,-2) circle [radius=0.06];
\draw [fill] (-2.5,-3) circle [radius=0.06];

\draw [thick] (3,0) --(7,0);
\draw [thick] (7,0) --(7,-2);
\draw [thick] (7,-2) --(6.5,-3);
\node [below] at (3,0) {$\mathbf{a}$};
\node [right] at (7,0) {$\mathbf{b}$};
\node [left] at (7,-2) {$\mathbf{c}$};
\node [below] at (6.5,-3) {$\mathbf{d}$};
\draw [fill] (3,0) circle [radius=0.06];
\draw [fill] (7,0) circle [radius=0.06];
\draw [fill] (7,-2) circle [radius=0.06];
\draw [fill] (6.5,-3) circle [radius=0.06];

\draw [thick] (9,0) --(13,0);
\draw [thick] (10,0) --(10,-2);
\draw [thick] (10,-2) --(10.5,-3);
\node [below] at (9,0) {$\mathbf{a}$};
\node [below] at (13,0) {$\mathbf{b}$};
\node [left] at (10,-2) {$\mathbf{c}$};
\node [below] at (10.5,-3) {$\mathbf{d}$};
\draw [fill] (9,0) circle [radius=0.06];
\draw [fill] (10,0) circle [radius=0.06];
\draw [fill] (13,0) circle [radius=0.06];
\draw [fill] (10,-2) circle [radius=0.06];
\draw [fill] (10.5,-3) circle [radius=0.06];

\end{tikzpicture}
\caption{cases $(1),(2)$ and $(3)$} {\label{fig:caterpillar2}}
\end{center}
\end{figure}

We have three cases:
\begin{enumerate}
\item $p(a,c) \cap p(a,b)= \{a\}$;
\item $p(a,c) \cap p(a,b)= p(a,b)$;
\item $\{a\} \subsetneq p(a,c) \cap p(a,b)\subsetneq p(a,b)$.
\end{enumerate}

Here we study the third case, the other ones are analogous. Call $d$ a vertex such that $p(c,d) \cap p(c,a) = \{c\}$ (see Figure \ref{fig:caterpillar2}), we have that
$$D_{a,b}({\cal C})+D_{c,d}({\cal C}) < D_{a,c}({\cal C})+D_{b,d}({\cal C})$$
and
$$D_{a,b}({\cal C})+D_{c,d}({\cal C}) < D_{a,d}({\cal C})+D_{b,c}({\cal C}),$$
which is absurd.
\end{proof}

Now we are ready to 
give a characterization of caterpillarlike families of positive real numbers:
\begin{thm}
Let $\{D_I\}_{I \in {[n] \choose 2}}$ be a family  in $\R_{+}$.
Call $a,b$ two elements of $[n]$ such that 
$$D_{a,b} -t_a -t_b = max_{i,j \in [n]} \{D_{i,j} -t_i -t_j \};$$ 
the family is caterpillarlike if and only if the following conditions hold:

\begin{itemize}
\item[$(i)$] the family satisfies  the four-point condition,
\item[$(ii)$] the family is median,
\item[$(iii)$] $D_{a,b}+D_{i,j}\geq max \{ D_{a,i}+D_{b,j}, D_{a,j}+D_{b,i}\}$ for any $i,j \in [n]\smallsetminus \{a,b\}$.
\end{itemize}
\end{thm}  

\begin{proof}
Suppose that the family $\{D_I\}_I$ is caterpillarlike, then there exists a positive-weighted caterpillar ${\cal C}=(C,w)$ with $V(C)=[n]$ realizing the family. Obviously the family must satisfy conditions $(i)$ and $(ii)$ and, by Remark \ref{condizionecaterpillar}, also condition $(iii)$ holds.

On the other hand, suppose to have a family of positive real numbers satisfying conditions $(i)$, $(ii)$ and $(iii)$. Since the family satisfies the four-point condition and is median, by Theorem \ref{thm:median} there exists a positive-weighted tree ${\cal C}=(C,w)$ with $V(C)=[n]$ realizing the family. Moreover, by condition $(iii)$ and Remark \ref{condizionecaterpillar}, $C$ is a caterpillar, as we wanted to prove. 
\end{proof}

Finally, we want to observe that, if we have a family of numbers in $\mathbb{R_+}$ which is caterpillarlike and  we call $a,b$ two elements of $[n]$ such that $D_{a,b} -t_a -t_b = max_{i,j \in [n]} \{D_{i,j} -t_i -t_j \}$, then it is easy to construct a positive-weighted caterpillar realizing the family: it is sufficient to draw a path of length $D_{a,b}$ with endpoints $a$ and $b$; for any $i \in [n] \smallsetminus \{a,b\}$, if $t_i >0$ attach a pendant edge with weight $t_i$ and leaf $i$ to a point of the path which has distance from $a$ equal to $D_{a,i}-t_i$; if $t_i =0$, call $i$ a vertex on the path which has distance from $a$ equal to $D_{a,i}.$ 


\section{Polygons}

Let ${\cal P}=(P,w)$ be 
a positive-weighted polygon  such that $V(P)=[n]$.

Observe that in case 
 ${\cal P}$ is not pruned, there is at most one useless edge $e$. So, if we delete $e$, we obtain a positive-weighted snake ${\cal {\tilde P}}=(\tilde{P},\tilde{w})$ with $V(\tilde{P})=[n]$ and with the same family of $2$-weights.

Suppose now that ${\cal P}$ is pruned: 
by Remark \ref{latounico},  for any $i \in V(P)$,
the vertices $x$ and $y$ adjacent to $i$ are exactly the ones
such that $D_{i,x}({\cal P})$ and $D_{i,y}({\cal P})$ are indecomposable
so it is possible to recover the order of the vertices of the polygon starting from the $2$-weights. 

\begin{defin}
Let $P$ be a polygon with $[n]$ as vertex set. We say that the vertex set  is sequentially  ordered if  $i$ and $i+1$ are adjacent for any $i \in [ n-1]$ and 
$n$ and $1$ are adjacent. 
\end{defin}

\begin{defin} \label{ordine} 
Let $\{D_I\}_{I \in {[n] \choose 2}}$ be a family  in $\R_{+}$
 satisfying the triangle inequalities and such that, for any $i \in [n]$, there exist exactly two elements $x,y \in [n] \smallsetminus \{i\}$ for which $D_{i,x}$ and $D_{i,y}$ are indecomposable. We can  rename the elements of $[n]$  according the following algorithm:

rename $1$ and $2$ two elements of $[n]$ such that $D_{1,2}=min\{D_{i,j}\}$ and define $H=\{1,2\}$.
Observe that $D_{1,2}$ must be indecomposable. 
  Rename $3$ the unique element in $[n] \smallsetminus \{1,2\}$ such that $D_{2,3}$ is indecomposable and put $3$ in $H$.  Recursively, call $i+1$ the unique element in $[n] \smallsetminus \{i-1,i\}$ such that $D_{i,i+1}$ is indecomposable. If $i+1 \in H$ stop the algorithm, otherwise put $i+1$ in $H$. 
\end{defin}

\begin{thm}\label{prunedpolygon}
Let $\{D_I\}_{I \in {[n] \choose 2}}$ be a family  in $\R_{+}$ satisfying the triangle inequalities; there exists a pruned positive-weighted polygon ${\cal P}=(P,w)$ with $V(P)=[n]$ realizing the family if and only if the following conditions hold:

\begin{itemize}
\item[$(i)$] for any $i \in [n]$ there are exactly two elements $x,y \in [n]$ such that $D_{i,x}$ and $D_{i,y}$ are indecomposable;
\item[$(ii)$] if $H$ is the set described in Definition \ref{ordine},  then the cardinality of $H$ is $n$;
\item[$(iii)$] if the elements of $[n]$ are renamed as in Definition  \ref{ordine}, then, for any $a<b \in [n]$, we have that
\begin{equation} \label{eq:1}
D_{a,b}=min\left\{ \sum_{i=a}^{b-1} D_{i,i+1}, \sum_{i=b}^{n-1} D_{i,i+1} + D_{1,n}+\sum_{i=1}^{a-1} D_{i,i+1} \right\}.
\end{equation}
\end{itemize}
\end{thm}

\begin{proof}
Suppose that there exists a pruned positive-weighted polygon ${\cal P}=(P,w)$ with $V(P)=[n]$ such that $D_{i,j}({\cal P})=D_{i,j}$ for any $i,j \in [n]$. It is easy to check that conditions $(i)$ and $(ii)$ hold. Moreover, if we rename the vertices as in Definition \ref{ordine}, the vertex set    is sequentially  ordered. Since ${\cal P}$ is pruned, for any $i \in [n-1]$ the $2$-weight $D_{i,i+1}({\cal P})$ is realized by  $e(i,i+1)$ and the $2$-weight $D_{1,n}({\cal P})$ is realized by  $e(1,n)$ (see Remark \ref{latounico}). Obviously, for any two vertices $a,b\in [n]$, with $a<b$, a subgraph realizing the $2$-weight $D_{a,b}({\cal P})$ is a  path with endpoints $a$ and $b$   and in the polygon there are exactly two different paths with endpoints $a$ and $b$. Their weights are the numbers at the second member of (\ref{eq:1}), so we get condition $(iii)$.

On the other hand, let $\{D_I\}_I$ be  a family of positive real numbers  satisfying conditions $(i)$,$(ii)$ and $(iii)$. By condition $(i)$ and $(ii)$ we can rename all the elements of  $[n]$ as in Definition \ref{ordine}.
Let  ${\cal P}=(P,w)$ be  the positive-weighted polygon with $V(P)=[n]$, 
with the vertex set sequentially ordered,   and  
 such that $w(e(i,i+1))=D_{i,i+1}$ for any $i \in [n-1]$ and $w(e(1,n))=D_{1,n}$ (see Figure \ref{figure2}).   
 
\begin{figure}[h!]
\begin{center}
\begin{tikzpicture}
\draw [thick] (0,0) --(1,0);
\draw [thick] (1,0) --(2,0.5);
\draw [thick] (0,0) --(-1,0.5);
\draw [thick] (2.7,1.3) --(2,0.5);
\draw [thick] (-1.7,1.3) --(-1,0.5);
\draw [thick] (2.7,1.3) --(3,2);
\draw [thick] (-1.7,1.3) --(-2,2);
\draw [thick] (3,2.7) --(3,2);
\draw [thick] (-2,2.7) --(-2,2);
\draw [thick] (1,4.7) --(2,4.2);
\draw [thick] (0,4.7) --(-1,4.2);
\draw [thick] (2.7,3.4) --(2,4.2);
\draw [thick] (-1.7,3.4) --(-1,4.2);
\draw [thick] (2.7,3.4) --(3,2.7);
\draw [thick] (-1.7,3.4) --(-2,2.7);
\draw [thick] (0,4.7) --(1,4.7);

\node [left] at (-2,2.7) {$\mathbf{1}$};
\node [right] at (-2,2.35) {$D_{1,2}$};
\node [left] at (-2,2) {$\mathbf{2}$};
\node [right] at (-1.85,1.75) {$D_{2,3}$};
\node [left] at (-1.7,1.3) {$\mathbf{3}$};
\node [right] at (-1.45,0.9) {$D_{3,4}$};
\node [left,below] at (-1.2,0.5) {$\mathbf{4}$};
\node [right] at (-0.6,0.35) {$D_{4,5}$};
\node [left,below] at (0,0) {$\mathbf{ \quad\quad 5  \quad . . .}$};
\node [right] at (-1.85,3.05) {$D_{1,n}$};
\node [left] at (-1.7,3.4) {$\mathbf{n}$};
\node [right] at (-1.4,3.7) {$D_{n,n-1}$};
\node [left] at (-1,4.2) {$\mathbf{n-1}$};
\draw [fill] (0,0) circle [radius=0.06];
\draw [fill] (1,0) circle [radius=0.06];
\draw [fill] (-1,0.5) circle [radius=0.06];
\draw [fill] (2,0.5) circle [radius=0.06];
\draw [fill] (-1.7,1.3) circle [radius=0.06];
\draw [fill] (2.7,1.3) circle [radius=0.06];
\draw [fill] (-2,2) circle [radius=0.06];
\draw [fill] (3,2) circle [radius=0.06];
\draw [fill] (3,2.7) circle [radius=0.06];
\draw [fill] (-2,2.7) circle [radius=0.06];
\draw [fill] (2.7,3.4) circle [radius=0.06];
\draw [fill] (-1.7,3.4) circle [radius=0.06];
\draw [fill] (2,4.2) circle [radius=0.06];
\draw [fill] (-1,4.2) circle [radius=0.06];
\draw [fill] (0,4.7) circle [radius=0.06];
\draw [fill] (1,4.7) circle [radius=0.06];
\end{tikzpicture}
\caption{ \label{figure2}}
\end{center}
\end{figure}

We have to prove that $D_{a,b}({\cal P})=D_{a,b}$ for any $a,b \in [n]$; 
obviously a subgraph realizing $D_{a,b}({\cal P})$ is  a  path with endpoints  $a$ and $b$
 and in the polygon there are exactly two different paths with endpoints  $a$ and $b$. By the definition of ${\cal P}$, their weights are the two numbers  at the  second member of (\ref{eq:1}), so we have that 
$$D_{a,b}({\cal P})=min\left\{ \sum_{i=a}^{b-1} D_{i,i+1}, \sum_{i=b}^{n-1} D_{i,i+1} + D_{1,n}+\sum_{i=1}^{a-1} D_{i,i+1} \right\}=D_{a,b},$$
where the last equality holds by  (\ref{eq:1}).
Observe that ${\cal P}$ is pruned, in fact, if an edge $e(a,b)$ (with $a$ and $b$ adjacent vertices)  were useless,   then $D_{a,b}({\cal P})$ would not be indecomposable, which is absurd because we have constructed ${\cal P}$ in such a way that two vertices are adjacent if and only if 
$D_{a,b}$ is indecomposable.
\end{proof}

Now we can give a characterization of the families of positive real numbers that are polygonlike:
\begin{thm}
A family of positive real numbers $\{D_I\}_{I \in {[n]  \choose 2}}$ satisfying the triangle inequalities is polygonlike if and only if either it is snakelike or it satisfies conditions $(i)$,$(ii)$ and $(iii)$ of Theorem \ref{prunedpolygon}.
\end{thm}  
\begin{proof}
Suppose there exists a positive-weighted polygon ${\cal P}=(P,w)$ with $V(P)=[n]$ realizing the family; if ${\cal P}$ is pruned, then by Theorem \ref{prunedpolygon} the family must satisfy conditions $(i)$,$(ii)$ and $(iii)$. If ${\cal P}$ is not pruned, then we can delete the unique useless edge  and we obtain a positive-weighted snake realizing the family, so the family is snakelike.

Conversely, suppose there exists a positive-weighted snake ${\cal S}=(S,w)$ with $V(S)=[n]$ realizing the family. If $i,j$ are the endpoints of the snake, we can add to the snake an edge $e(i,j)$ with weight any real number greater than or equal to $D_{i,j}$: it is easy to check that
the  positive-weighted polygon with $n$ vertices we have obtained realizes the family $\{D_{I}\}_I$, so the family is also polygonlike. Finally, if the family satisfies conditions $(i)$,$(ii)$ and $(iii)$ of Theorem \ref{prunedpolygon},  it is polygonlike by Theorem  \ref{prunedpolygon}.
\end{proof}

\section{Complete graphs and bipartite graphs}

An immediate consequence of  Remark \ref{latounico} is the following characterization of the co-graphlike families of $2$-weights:

\begin{rem}\label{thm:complete}
Let $\{D_I\}_{I \in {[n] \choose 2}}$ be a family  in $\R_{+}$ satisfying the triangle inequalities; the family is co-graphlike if and only if $D_{i,j}$ is indecomposable for any $i,j \in [n]$. 
\end{rem}

\begin{proof}
Suppose there exists a pruned positive-weighted complete graph ${\cal G}=(G,w)$ with $V(G)=[n]$ realizing the family; thus $e(i,j)$ is useful for any $i,j \in [n]$; so, by Remark \ref{latounico}, $D_{i,j}({\cal G})$ is indecomposable for any $i,j \in [n]$. On the other hand,  suppose  $D_{i,j}$ is indecomposable for any $i,j \in [n]$; 
  by  Theorem \ref{HY}, there exists a positive-weighted graph ${\cal G}=(G,w)$ with $V(G)=[n]$ realizing the family; moreover, since $D_{i,j}$ is indecomposable for any $i,j \in [n]$, then, by Remark \ref{latounico}, we have that $e(i,j) \in E(G)$ and $e(i,j)$ is useful for any $i,j \in [n]$. 
\end{proof}

Now we want to characterize  bigraphlike families of positive real numbers; first of all,  given a positive-weighted  bipartite graph ${\cal G}=(G,w)$ on $X,Y \subset V(G)$, we show that   it is possible to recover $X$ and $Y$  from the family of $2$-weights of ${\cal G}$.

\begin{prop}\label{charsecondomodo}
Let ${\cal B}=(B,w)$ be a positive-weighted bipartite graph on $X$ and $Y$ with $V(B)=[n]$; let  $x \in X$ and $y \in Y$;  then:

$$X=\{x\} \cup \Big\{ i \in [n]\smallsetminus\{x,y\} \; | \; \exists \, j_1,...,j_t \in [n] \textrm{  with $t$ odd such that }$$
\begin{equation}\label{eq:insieme2}
D_{x,i}({\cal B})=D_{x,j_1}({\cal B})+D_{j_1,j_2}({\cal B})+...+  D_{j_t,i}({\cal B})\textrm{  and the elements of the sum are indecomposable }\Big\}
\end{equation}
and
$$Y=\Big\{ i \in [n]\smallsetminus\{x\} \; |   \textrm{ either  } D_{x,i} \textrm{ is indecomposable or } \, \exists \, j_1,...,j_t \in [n] \textrm{  with $t$ even such that }$$
\begin{equation}\label{eq:insieme3}
D_{x,i}({\cal B})=D_{x,j_1}({\cal B})+D_{j_1,j_2}({\cal B})+...+  D_{j_t,i}({\cal B})\textrm{  and the elements of the sum are indecomposable }\Big\}.
\end{equation}
\end{prop}

\begin{proof}
Let us prove (\ref{eq:insieme2}); the other equality can be proved analogously. Call $R$ the second member of (\ref{eq:insieme2}); we want to prove that $X=R$. 
\begin{itemize}
\item[-] $X \subset R:$ let $i \in X \smallsetminus\{x\}$; 
observe that  $D_{x,i}(\cal B)$ is not indecomposable: otherwise 
by Remark \ref{latounico}, we would have $e(x,i) \in E(B)$, which is absurd; so we can
write $D_{x,i}(\cal B)$ as
$$D_{x,j_1}({\cal B})+D_{j_1,j_2}({\cal B})+...+  D_{j_t,i}({\cal B})$$ for some $j_1,..., j_t$ with 
$D_{x,j_1}({\cal B}),D_{j_1,j_2}({\cal B}),....,  D_{j_t,i}({\cal B})$ indecomposable. By 
Remark \ref{latounico}, the $2$-weights $D_{x,j_1}({\cal B}),D_{j_1,j_2}({\cal B}),....,  D_{j_t,i}({\cal B})$ are realized respectively by 
$e(x,j_1), e(j_1,j_2), ....., e(j_t,i) $; thus the path given by the union of these edges realizes   $D_{x,i}(\cal B)$ and, since $x,i \in X$, we have that   $t$ is necessarily  odd.

\item[-] $R \subset X:$ if $i \in R$ then there exist $j_1,...,j_t \in [n]$ with $t$ odd such that $D_{x,i}({\cal B})=D_{x,j_1}({\cal B})+D_{j_1,j_2}({\cal B})+...+  D_{j_t,i}({\cal B})$ and the elements of the sum are indecomposable. By Remark \ref{latounico}, the $2$-weights $D_{x,j_1}({\cal B}),D_{j_1,j_2}({\cal B}),... D_{j_t,i}({\cal B})$ are realized respectively only by the edges $e(x,j_1)$, $e(j_1,j_2)$,...,$e(j_t,i)$, which implies that $D_{x,i}({\cal B})$ is realized by the path given by these edges; so, since $t$ is odd, $i \in X$. 
\end{itemize}
\end{proof}

\begin{rem} 
Let ${\cal B}=(B,w)$ be a positive-weighted bipartite graph on $X$ and $Y$ with $V(B)=[n]$. Let $x,y \in [n]$ be such that 
 $$D_{x,y}({\cal B})=min_{1 \leq i<j\leq n}\{D_{i,j}({\cal B})\};$$ 
hence, obviously, $D_{x,y}({\cal B})$ is indecomposable, and then,
by Remark \ref{rem2},  $e(x,y) \in E(B)$ and 
 $D_{x,y}({\cal B})$ is realized only by the path with unique edge
 $e(x,y)$.
\end{rem}

Now  we are ready to give a characterization of the families of positive real numbers that are bigraphlike:

\begin{thm}
Let $\{D_I\}_{I \in {[n] \choose 2}}$ be a family  in $\R_{+}$ satisfying the triangle inequalities and let $x,y \in [n]$ such that  $D_{x,y}=min_{1 \leq i<j\leq n}\{D_{i,j}\}$; define
$$X=\{x\} \cup \Big\{ i \in [n]\smallsetminus\{x,y\} \; | \; \exists \, j_1,...,j_t \in [n] \textrm{  with $t$ odd such that }$$
$$D_{x,i}=D_{x,j_1}+D_{j_1,j_2}+...+  D_{j_t,i}\textrm{  and the elements of the sum are indecomposable }\Big\}$$
and
$$Y=\Big\{ i \in [n]\smallsetminus\{x\} \; |   \textrm{ either }  \, D_{x,i} \textrm{ is indecomposable or } \, \exists \, j_1,...,j_t \in [n] \textrm{  with $t$ even such that }$$
$$D_{x,i}=D_{x,j_1}+D_{j_1,j_2}+...+  D_{j_t,i}\textrm{  and the elements of the sum are indecomposable }\Big\}.$$
The family $\{D_I\}_I$ is bigraphlike if and only if the following conditions hold:

(1) $X \cap Y = \emptyset$ 

(2) for any $a,b \in X$ (respectively $Y$), there exists $z \in Y$ (respectively $X$) such that:
$$D_{a,b}=D_{a,z}+D_{z,b}.$$
\end{thm}  

\begin{proof}
Suppose there exist two subset $X$ and $Y$ of $[n]$ and  a positive-weighted bipartite graph ${\cal B}=(B,w)$ on $X'$ and $Y'$  with $V(B)=[n]$ realizing the family. By Proposition \ref{charsecondomodo} we have that $X=X'$ and $Y=Y'$ (or vice versa), so $X \cap Y = \emptyset$.  Let $a,b \in X$; a path realizing $D_{a,b}({\cal B})$ must contain a vertex $z \in Y$, so, by Remark \ref{rem2}, we have that  $D_{a,b}({\cal B})=D_{a,z}({\cal B})+D_{z,b}({\cal B}).$ If both $a$ and $b$ are elements of $Y$, the proof is analogous.

Now, suppose that $\{D_I\}_I$  satisfyies (1) and (2). 
Let ${\cal B}=(B,w)$
be the positive-weighted bipartite graph  on $X$ and $Y$ such that:
\begin{itemize}
\item[-] $V(G)=[n]$;
\item[-] $E(G)=\{(a,b) \, | \, a \in X,\, b \in Y\}$;
\item[-] $w(e(a,b))=D_{a,b}$ for any $a \in X$ and $b \in Y$.
\end{itemize}

We want to prove that $D_{a,b}({\cal B})=D_{a,b}$ for any $a,b \in [n]$.
Let $p$ be a path  realizing $D_{a,b}({\cal B})$ and let
$j_1,.., j_t \in [n]$ be such that   $p$ is given by 
 $e(a,j_1), e(j_1,j_2),...,e(j_t,b)$; then  we  have that
\begin{equation} \label{geq}
D_{a,b}({\cal B}) = w(e(a,j_1))+w(e(j_1,j_2))+...+w(e(j_t,b))=D_{a,j_1}+D_{j_1,j_2}+...+D_{j_t,b} \geq D_{a,b},
\end{equation}
where the last inequality follows from the triangle inequalities.

 If $a \in X$ and $b \in Y$ (or vice versa), then 
\begin{equation} \label{leq1}
D_{a,b}({\cal B}) \leq w(e(a,b))=D_{a,b},
\end{equation}
  so, from (\ref{geq}) and (\ref{leq1}), we get 
$D_{a,b}({\cal B}) =D_{a,b}$.

If both $a$ and $b$ are in $X$ (if they are in $Y$, we can argue analogously), by assumption, there exists $z \in Y$ such that $D_{a,z} + D_{b,z} = D_{a,b};$ we have that 
\begin{equation} \label{leq2}
D_{a,b}({\cal B}) \leq D_{a,z} + D_{b,z} = D_{a,b},
\end{equation}
where the inequality  holds
because the path given by $e(a,z) $ and $ e(z,b)$ contains $a$ and $b$ as vertices and its weight 
is equal to $D_{a,z} + D_{b,z}$; 
 so, from (\ref{geq}) and (\ref{leq2}), we get
  also in this case  $D_{a,b}({\cal B}) =D_{a,b}$.
\end{proof}

Finally we give also a characterization of  co-bigraphlike families of positive real numbers:

\begin{rem}\label{thm:complete}
Let $\{D_I\}_{I \in {[n] \choose 2}}$ be a family  in $\R_{+}$
 which is bigraphlike on $X,Y \subset [n]$. The family is co-bigraphlike on $X$ and $Y$  if and only if $D_{i,j}$ is indecomposable for any $i \in X,j \in Y$. 
\end{rem}

\begin{proof}
Let ${\cal B}=(B,w)$ be a  positive-weighted complete bipartite graph on $X$ and $Y$, with $V(B)=[n]$, realizing the family.
If it is pruned, then  
 $e(i,j)$ is useful for any $i \in X, j \in Y$; so, by Remark \ref{latounico}, $D_{i,j}({\cal B})$ is indecomposable for any $i \in X, j \in Y$. 
 On the other hand, 
 if  $D_{i,j}$ is indecomposable for any $i \in X$ and $j \in Y$, then, by Remark \ref{latounico}, $e(i,j) \in E(B)$ and $e(i,j)$ is useful for any $i \in X, j \in Y$. 
\end{proof}

\section{Planar graphs}

\begin{defin} Let $G$ be a graph and let $e(u,v)$ be an edge of $G$.
We say that a graph $G'$ is obtained from $G$ by a subdivision of 
 the edge $ e(u,v)$ if $V(G') $ is the union of $V(G)$ and a new vertex $z$  and $E(G')$ is $E(G)- \{e(u,v)\} \cup  \{e(u,z),e(z,v)\}$.
We say that a graph $G'$ is a subdivision of a graph $G$ if it is the
 graph resulting from the subdivision of some edges in $G$. 
\end{defin}

\begin{thm} \label{Kur}
{\bf (Kuratowski)}
A finite graph is planar if and only if it does not contain a subgraph that is a subdivision of $K_5$  or of $K_{3,3}$.  
\end{thm}

\begin{defin} Let $G$ be a subdivision of $K_5$. We say that a vertex 
of $G$ is a \textbf{true vertex} if it is a vertex of  $K_5$. We call \textbf{verges} of $G$ the paths that are subdivisions of the edges of $K_5$.
\end{defin}

\begin{prop} \label{belliipruned}
Let ${\cal G}=(G,w)$ be a positive-weighted graph with $V(G)=[n]$.
Suppose it is pruned. Let us denote $D_{i,j}({\cal G})  $ by $D_{i,j} $ for any $i,j \in [n]$.

(i) $G$ contains a subdivision of $K_5$ $\Longleftrightarrow$ 
there exists $Q \in {[n] \choose 5}$ such that for any distinct  $a,b
\in Q$, 
either $D_{a,b}$ is indecomposable 
or 
there exists a  sequence $(x_1,...., x_r)$ in $[n]-Q$ (depending on $\{a,b\}$) such that $D_{a,x_1},....,D_{x_r,b}$  are indecomposable and, if $\{a,b\} \neq \{a',b'\}$, the sequence of $\{a,b\}$ does not intersect the sequence of $\{a',b'\}$.

(ii) $G$ contains a subdivision of $K_{3,3}$ $\Longleftrightarrow$ 
there exist disjoint $A,B \in {[n] \choose 3}$ such that for any  $a \in A$ and $b \in B$, 
either $D_{a,b}$ is indecomposable 
or 
there exists a  sequence $(x_1,...., x_r)$ in $[n]-A-B$ (depending on $\{a,b\}$) such that $D_{a,x_1},....,D_{x_r,b}$  are indecomposable and, if $\{a,b\} \neq \{a',b'\}$, the sequence of $\{a,b\}$ does not intersect the sequence of $\{a',b'\}$.
\end{prop}

\begin{proof} Let us prove (i) 
 (the proof of (ii) is analogous).
 
$\Longleftarrow$ By Remark \ref{latounico}, if $D_{i,j}$ is indecomposable, then
$e(i,j) \in E(G)$. For any $a,b \in Q$, let $c_{a,b}$ be the following path:
the path given only by the edge $e(a,b)$ if $D_{a,b}$ is indecomposable, the path
given by the edges $e(a,x_1)$, $e(x_1,x_2)$, ....., $e(x_r,b)$ if $(x_1,..., x_r)$ is a sequence as in the statement of the proposition.

 The union of the paths $c_{a,b}$ for $a,b \in Q$ gives a subgraph that is a subdivision of $K_5$.

$ \Longrightarrow$  Let $G'$ be a  subdivision of $K_5$ in $G$.
Let $Q$ be the set of the true vertices of $G'$. Since 
${\cal G}$ is pruned, every edge is useful, in particular,
for any $x,y \in [n]$ such that $e(x,y)$ is  in a  verge of $G'$, we have that  $e(x,y)$ is useful,   so, by Remark \ref{latounico}, the $2$-weight $D_{x,y}$ is indecomposable and then we get our statement.
\end{proof}

\begin{thm}
Let $\{D_I\}_{I \in {[n]  \choose 2}}$ be a family in $\mathbb{R_+}$. It is planargraphlike if and only if  the following conditions hold:

(a) the family satisfies the triangle inequalities;

(b) there does not exist $ Q \in {[n] \choose 5}$ such that, for any distinct  $a,b \in Q$, 
either $D_{a,b}$ is indecomposable 
or 
there exists a  sequence $(x_1,...., x_r)$ in $[n]-Q$ (depending on $\{a,b\}$) such that $D_{a,x_1},....,D_{x_r,b}$  are indecomposable 
and, if $\{a,b\} \neq \{a',b'\}$, the sequence of $\{a,b\}$ does not intersect the sequence of $\{a',b'\}$;

(c) there do not exist disjoint  $ A,B \in {[n] \choose 3}$ such that,  for any  $a \in A$ and $b \in B$, 
either $D_{a,b}$ is indecomposable 
or 
there exists a  sequence $(x_1,...., x_r)$ in $[n]-A-B$ (depending on $\{a,b\}$) such that $D_{a,x_1},....,D_{x_r,b}$  are indecomposable and, if $\{a,b\} \neq \{a',b'\}$, the sequence of $\{a,b\}$ does not intersect the sequence of $\{a',b'\}$.
\end{thm}

\begin{proof} $\Longleftarrow$
Let ${\cal G}$ be a positive-weighted planar graph realizing the family.
By eliminating a useless edge, then another one and so on, we get 
a pruned positive-weighted planar graph realizing the family.  So we can 
conclude by Proposition \ref{belliipruned} and Theorem \ref{Kur}.

$\Longrightarrow$ 
By condition (a) and Theorem \ref{HY}, there exists a positive-weighted 
graph ${\cal G}$ realizing the family; 
by eliminating a useless edge, then another one and so on, we get 
a pruned positive-weighted planar graph realizing the family; 
by conditions (b) and (c) and  using Proposition \ref{belliipruned} and Theorem \ref{Kur}, we can conclude that it is planar.
\end{proof}

{\bf Acknowledgemnts.}
This work was supported by the National Group for Algebraic and Geometric Structures, and their  Applications (GNSAGA-INdAM). 
The first author was supported by Ente Cassa di Risparmio di Firenze.

\end{document}